\documentclass[11 pt, reqno]{amsart}

\usepackage[latin1]{inputenc}
\usepackage{amsmath,amsthm,amssymb,amscd}
\usepackage{bbm}

\newtheorem{thm}{Theorem}[section]
 
 \newtheorem{conj}[thm]{Conjecture}
 \newtheorem{lem}[thm]{Lemma}
 \newtheorem{prop}[thm]{Proposition}

 \theoremstyle{definition}
 \newtheorem{df}[thm]{Definition}
 \theoremstyle{remark}
 \newtheorem{rem}[thm]{Remark}
 
 \numberwithin{equation}{section}

\def\be#1 {\begin{equation} \label{#1}}
\newcommand{\ee}{\end{equation}}

\def\sqw{\hbox{\rlap{\leavevmode\raise.3ex\hbox{$\sqcap$}}$%
\sqcup$}}
\def\findem{\ifmmode\sqw\else{\ifhmode\unskip\fi\nobreak\hfil
\penalty50\hskip1em\null\nobreak\hfil\sqw
\parfillskip=0pt\finalhyphendemerits=0\endgraf}\fi}

\title{Strichartz estimates for the Schr\"odinger equation on irrational tori}

\author{Yu Deng}
\address[Y. Deng]{Courant Institute of Mathematical Sciences\\ New York University\\ 251 Mercer Street\\ New York, N.Y. 10012-1185\\ USA}
\email{yudeng@cims.nyu.edu}

\author{Pierre Germain}
\address[P. Germain]{Courant Institute of Mathematical Sciences\\ New York University\\ 251 Mercer Street\\ New York, N.Y. 10012-1185\\ USA}
\email{pgermain@cims.nyu.edu}

\author{Larry Guth}
\address[L. Guth]{Department of Mathematics \\ Massachussets Institute of Technology \\ 182 Memorial Drive \\ Cambridge, MA 02139}
\email{larry.guth.work@gmail.com}

\begin{document}

\maketitle

\begin{abstract}
We prove Strichartz estimates over large time scales for the Schr\"odinger equation set on irrational tori. They are optimal for Lebesgue exponents $p > 6$.
\end{abstract}

\section{Introduction} 

\subsection{Strichartz estimates on compact manifolds}  
The classical Strichartz estimates for the Schr\"odinger equation set in the Euclidean space $\mathbb{R}^d$ read (see~\cite{KT})
$$
\| e^{it\Delta} f(x) \|_{L^p_t (-\infty,\infty,L^q_x(\mathbb{R}^d))} \lesssim \| f \|_{L^2} \qquad \mbox{for $\frac{2}{p}+\frac{d}{q} = \frac{d}{2}, \;\; p \geq 2, \;\; (p,q) \neq (2,\infty)$},
$$
implying in particular (by Sobolev embedding) that, for $p \geq \frac{2(d+2)}{d}$,
$$
\| e^{it\Delta} f(x) \|_{L^p_{t,x} ((-\infty,\infty) \times \mathbb{R}^d)} \lesssim N^{\frac{d}{2}-\frac{d+2}{p}} \| f \|_{L^2} \qquad \mbox{if $\operatorname{Supp} \widehat{f} \subset B(0,N)$}.
$$
Given a compact Riemannian manifold $M$, with Laplace-Beltrami operator $\Delta$, and associated Sobolev spaces $H^s$, it is natural to ask for similar estimates on the Schr\"odinger group $e^{it\Delta}$: what is the best constant $C(M,p,T,N)$ in
$$
\|e^{it\Delta} f\|_{L^{p}_{t,x}([0,T] \times M)} \leq C(M,p,T,N) \| f \|_{L^2}
$$
(under some spectral assumption generalizing the Fourier support condition $\operatorname{Supp} \widehat{f} \subset B(0,N)$)? 

Little is known about this question for general manifolds, but an upper bound on $C(M,p,T,N)$ was derived by Burq, G\'erard and Tzvetkov~\cite{BGT,BGT2}, which turns out to be sharp in some range for the sphere $\mathbb{S}^d$, at least if $d=3$. It was then showed~\cite{Thomann,ACMT} that the presence of a stable closed geodesic leads to a behavior similar to that of the sphere.

For tori, this question was recently answered by Bourgain and Demeter~\cite{BD} for time intervals $T \leq 1$.  They proved that if $R$ is a rectangular torus
$$
R = [0,\ell_1] \times \dots \times [0,\ell_d]
$$
(with the usual metric), then the following inequality holds:
for $p \geq 1$, $N \geq 1$, and for $\epsilon > 0$
\begin{equation}
\label{strichartz1}
\|e^{it\Delta} f\|_{L^{p}_{t,x}([0,1] \times R)} \lesssim _{\varepsilon}N^{\varepsilon} \left(1 + N^{\frac{d}{2}-\frac{d+2}{p}}\right)\|f\|_{L^2}\qquad\textrm{if }\operatorname{Supp} \widehat{f} \subset B(0,N).
\end{equation}
The question of Strichartz estimates on tori was first addressed by Bourgain~\cite{B1}, and this was followed by a number of works improving its results~\cite{B2,B3,CW,GOW}.  Before Bourgain and Demeter's paper, however, the sharp estimate seemed out of reach.  Also, before this paper, the estimates known for the cubic torus were better than for irrational tori, because of some number theoretic facts which were used in the arguments.  In this paper, we will study what happens for long time intervals $T >> 1$.  Over long time intervals, we will see that the behavior on irrational tori is actually better than the behavior on rational tori.

\subsection{Generic tori and times $T \geq 1$} Since the linear Schr\"odinger equation conserves the $L^2$ norm, the above estimate implies immediately, for $p \geq 1$, $N \geq 1$, $T \geq 1$, and for $\epsilon > 0$
\begin{equation}
\label{strichartz2}
\|e^{it\Delta} f\|_{L^{p}_{t,x}([0,T] \times R)} \lesssim _{\varepsilon}N^{\varepsilon} \left(1 + N^{\frac{d}{2}-\frac{d+2}{p}}\right) T^{\frac{1}{p}} \|f\|_{L^2}\qquad\textrm{if }\operatorname{Supp} \widehat{f} \subset B(0,N).
\end{equation}
This estimate is clearly optimal on the square torus, where the linear Schr\"odinger flow is periodic; but on irrational tori it raises the following question:
\textbf{For a generic choice of the parameters $(\ell_i)$, what is (up to sub-polynomial factors) the best constant $C(p,N,T)$ such that}
$$
\| e^{it\Delta} f \|_{L^p ([0,T] \times R)} \leq C(\ell_i,p,N,T) \| f \|_{L^2}, \qquad \mbox{if $\operatorname{Supp} \widehat{f} \subset B(0,N)$}?
$$
We will answer this question for $p > 6$ and obtain some upper and lower bounds for other $p$.

\subsection{Reformulation} Here we perform a change of variables to transform the problem to the square torus $\mathbb{T}^d=[0,1]^d$ (with periodic boundary conditions). Define the quadratic form
$$
Q (n_1,\dots,n_d) = \beta_1n_1^2 + \beta_2 n_2^2 + \cdots + \beta_d n_d^2, \quad \text{where $\beta_i=\ell_i^{-2}$,}
$$ 
and the corresponding differential operator 
$$
\Delta_\beta = \frac{1}{2\pi} Q (\partial_1,\dots,\partial_d) = \frac{1}{2\pi}\big(\beta_1\partial_1^2 + \beta_2 \partial_2^2 + \dots + \beta_d \partial_d^2\big).
$$ 
For a function $f$ defined on $R$, observe that
$$
(e^{it\Delta} f)(\ell_1x_1,\ell_2 x_2,\dots,\ell_d x_d) = 
(e^{2\pi it\Delta_\beta} g)(x_1,\dots,x_d) \quad \mbox{where $g(x_1,\dots,x_d) = f(\ell_1x_1,\ell_2 x_2,\dots,\ell_d x_d)$}.
$$
Therefore, Conjecture \ref{longstr} is equivalent to the corresponding estimates for the quantity\[\left\| e^{it\Delta_\beta} f \right\|_{L^p ([0,T] \times \mathbb{T}^d)}\quad\text{instead of}\quad \left\|e^{it\Delta}f\right\|_{L^p([0,T]\times R)}.\]
Moreover, the transformation between $(\ell_i)$ and $(\beta_i)$ is a diffeomorphism with positive Jacobian (so it maps null sets to null sets), thus below we will focus on the study of the quantity $\| e^{it\Delta_\beta} f\|_{L^p ([0,T] \times \mathbb{T}^d)}$ with parameters $(\beta_i)$.

\subsection{Genericity} To fix ideas, we will assume in the rest of this article that
$$
\beta_i \in [1,2] \quad \forall i \in \{1,\dots,d\}.
$$
\begin{df}
We will call a property \textit{generic} in $(\beta_1,\dots,\beta_d)$ if it is true for all $(\beta_1,\dots,\beta_d)$ outside of a null set (set with measure zero) of $[1,2]^d$.
\end{df}
Genericity will often be for us a consequence of a classical result on Diophantine approximation: it is well-known (see \cite{Cassels}) that, generically in $(\beta_i)$, there exists $C$ such that
\begin{equation}
\tag{D1} \label{dioph1}
\left|k_1+\beta_2k_2+\cdots+\beta_dk_d\right|\geq C \frac{1}{(|k_1| + \dots + |k_d|)^{d-1} \log(|k_1| + \dots + |k_d|)^{2d}}.
\end{equation}
A sharper version of this inequality, also true generically in $(\beta_i)$, is 
\begin{equation}
\tag{D2} \label{dioph2}
\left|k_1+\beta_2k_2+\cdots+\beta_dk_d\right|\geq C \prod_{i=2}^{d}(1+|k_i|)^{-1}(\log(2+|k_i|))^{-2}.
\end{equation}

\subsection{The conjecture} We propose the following conjecture. 
\begin{conj}\label{longstr} Let $d\geq 2$. For generic $(\beta_1,\cdots,\beta_d)$, one has  
\begin{equation}
\label{strichartzlong}\left\|e^{it\Delta_\beta}f\right\|_{L^p([0,T]\times \mathbb{T}^d)}\lesssim_{\varepsilon}N^{\varepsilon}\left(1 + N^{\frac{d}{2}-\frac{d+2}{p}} \right) \left[1+\left(\frac{T}{N^{\theta(p)}}\right)^{\frac{1}{p}}\right]\|f\|_{L^2} \qquad\textrm{if }\operatorname{Supp} \widehat{f} \subset B(0,N)
\end{equation}
for $N\geq 1$, $T\geq 1$ and arbitrarily small $\varepsilon>0$, where\begin{equation}\label{expo}\theta(p)=\left\{\begin{aligned}&0,&p\in&\left[1,\frac{2(d+2)}{d} \right),\\
&\frac{d}{2} \left( p - \frac{2(d+2)}{d} \right), &p\in&\left[\frac{2(d+2)}{d},6\right),\\
& 2d-2,&p\in&\left[ 6,\infty\right). \end{aligned}\right.\end{equation} Moreover, these estimates are sharp up to $N^{\varepsilon}$ losses for arbitrarily small $\varepsilon$.
\end{conj}

Notice first that for $T = 1$ or $p \leq \frac{2(d+2)}{d}$, it follows from the estimate of Bourgain an Demeter~\eqref{strichartz1}.

Let us now explain briefly why this conjecture is plausible. It can be written equivalently: for generic $\beta$, and for $T \geq 1$, 
\begin{equation}
\label{conj2}
\left\|e^{it\Delta_\beta} f \right\|_{L^p([0,T]\times \mathbb{T}^d )}\lesssim_{\varepsilon}N^{\varepsilon} \left[ T^{\frac{1}{p}} + N^{\frac{d}{2}-\frac{d+2}{p}} + N^{\frac{d}{2}-\frac{3d}{p}} T^{ \frac{1}{p}} \right]\| f\|_{L^2} \qquad\textrm{if }\operatorname{Supp} \widehat{f} \subset B(0,N).
\end{equation}
We will show heuristically that two simple examples saturate, in different regimes, the different terms in the right-hand side of the above. These two simple examples are $f \equiv 1$ and a peaked function living on a scale $\sim \frac{1}{N}$, such as 
$$
\psi(x) = \frac{1}{N^{d/2}} \sum_{n \in \mathbb{Z}^d} \chi \left( \frac{n_1}{N} \right) \dots \chi \left( \frac{n_d}{N} \right)  e^{2\pi i n \cdot x},
$$
where $\chi \in \mathcal{C}_0^\infty$. 

\bigskip

\noindent \underline{An example such that $\left\|e^{it\Delta_\beta} f \right\|_{L^p([0,T]\times R)} \sim T^{\frac{1}{p}} \| f \|_{L^2}$.} It obviously suffices to choose $f \equiv 1$, or $f(x) = e^{2\pi i n \cdot x}$ for some $n \in \mathbb{Z}^d$.

\bigskip

\noindent \underline{An example such that $\left\| e^{it\Delta_\beta} f \right\|_{L^p([0,T]\times R)} \sim N^{\frac{d}{2} - \frac{d+2}{p}} \| f \|_{L^2}$.} It is classical that $f = \psi$ satisfies this requirement as soon as $p \geq \frac{2(d+2)}{d}$, $T \geq 1$ (actually, $T \geq \frac{1}{N}$ suffices). This example shows the optimality of the result of Bourgain and Demeter~\eqref{strichartz1}.

\bigskip

\noindent \underline{An example such that $\left\| e^{it\Delta_\beta} f \right\|_{L^p([0,T]\times R)} \sim N^{\frac{d}{2} - \frac{3d}{p}}  T^{\frac{1}{p}} \| f \|_{L^2}$}. We now argue \underline{heuristically} that this should be the case for $f = \psi$, $p> \frac{2(d+2)}{d}$, $T > N^{2d-2}$. The main idea is that a significant contribution is made to the $L^p$ norm around the times $t_i$ where $u$ "refocuses", which is to say $u(t_i) \sim \psi$. These times occur with a period $\sim N^{2d-2}$, and around each time $t_i$, the contribution is of order $\sim N^{\frac{d}{2} - \frac{d+2}{p}}$ (by the previous paragraph). Therefore the $L^p$ norm on $[0,T] \times \mathbb{T}^d$ is of order $\left( \frac{T}{N^{2d-2}} \right)^{\frac{1}{p}} N^{\frac{d}{2} - \frac{d+2}{p}} \sim N^{\frac{d}{2} - \frac{3d}{p}}$.

We now explain why the time needed for the wave $u(t) = e^{it \Delta_\beta} \psi(x)$ to refocus is of order $N^{2d-2}$.
Observe that $u$ can be written
\begin{equation}
\label{equ}
u(t) = e^{it \Delta_\beta}  \psi(x) = \frac{1}{N^{d/2}} \sum_{n \in \mathbb{Z}^d} \chi \left( \frac{n}{N} \right) e^{2\pi i (n \cdot x + tQ_\beta(k))}.
\end{equation}
At the initial time, $u(t=0) = \psi(x)$, which is a very peaked function living on a scale $\sim \frac{1}{N}$. How long does it take before the wave $u$ "refocuses"? Fixing $\eta>0$, we argue that there is a time $t = q \sim N^{2d-2+\eta}$ such that $u(q) \sim \psi$. Indeed, by classical (simultaneous) Diophantine approximation theory~\cite{Cassels}, there exists, for generic $(\beta_i)$, an integer $q \sim N^{2d-2+\eta}$ such that, for all $i \in \{ 2,\dots,d\}$,
$$
\beta_i = \frac{p_i}{q} + O\left( \frac{1}{q N^{2+\frac{\eta}{d-1}}} \right) \qquad \mbox{for all $i \in \{ 2,\dots,d\}$}.
$$
This implies that, for $|k| \lesssim N$,
$$
\left\| q Q_\beta(k) \right\| = \left\| q \sum_{i=1}^d \beta_i |k_i|^2 \right\| \lesssim \frac{N^2}{N^{2+\frac{\eta}{d-1}}} << 1,
$$
where, for a real number $x$, we denote $\| x \|$ for the distance from $x$ to the closest integer.
Coming back to~\eqref{equ}, this implies that $u(t=q) \sim \psi$.

\subsection{Obtained result: optimality of the conjecture}
Our first result gives the optimality of the conjecture. We saw above very simple examples such that $\left\|e^{it\Delta_\beta} f \right\|_{L^p([0,T]\times R)} \sim T^{\frac{1}{p}} \| f \|_{L^2}$ and $\left\| e^{it\Delta_\beta} f \right\|_{L^p([0,T]\times R)} \sim N^{\frac{d}{2} - \frac{2(d+2)}{p}} \| f \|_{L^2}$. This implies that the two first terms on the right-hand side of~\eqref{conj2} are necessary. The third term, namely $N^{\frac{d}{2}-\frac{3d}{p}} T^{\frac{1}{p}}$, becomes dominant in the range $p>6$, and was justified heuristically above. We now provide a rigorous statement.

\begin{thm} For any $\eta>0$, and provided $\beta_2, \dots , \beta_d$ are generic, there exists $f$ such that $\operatorname{Supp} \widehat f \subset B(0,N)$, and, if $T > N^{2d-2+\eta}$
$$
\| e^{it \Delta_\beta} f \|_{L^p([0,T] \times \mathbb{T}^d)} \sim N^{\frac{d}{2} - \frac{3d}{p}} T^{\frac{1}{p}} \| f \|_{L^2(\mathbb{T}^d)}.
$$
\end{thm}
This theorem is proved in Section~\ref{smt0}.

\subsection{Obtained result: partial proof of the conjecture}

\begin{thm}\label{mainthmfin} Conjecture \ref{longstr} is true for $p>6$, and a weaker version holds for $p<6$. More precisely, the inequality~\eqref{strichartzlong}, holds, up to sub-polynomial factors, with $\theta(p)$ given by 
\[
\mbox{if $d=2$}, \;\;
\theta(p)=\left\{\begin{aligned}&0,&p\in&\left[1,4\right),\\
&\frac{2(p-4)}{p+4},&p\in&\left[4,6\right),\\
&2,&p\in&\left[6,\infty\right),\end{aligned}\right.
\qquad \quad
\mbox{if $d=3$}, \;\;
\theta(p)=\left\{\begin{aligned}&0,&p\in&\left[1, \frac{10}{3} \right),\\
&\frac{4(3p-10)}{3p+14},&p\in&\left[ \frac{10}{3},6\right),\\
&4,&p\in&\left[6,\infty\right),\end{aligned}\right.
\]
and
\[
\mbox{if $d\geq 4$}, \qquad
\theta(p)=\left\{\begin{aligned}&0,&p\in&\left[1,\frac{2(d+2)}{d} \right),\\
&\frac{2(d-1)(pd-2d-4)}{pd+6d-4},&p\in&\left[\frac{2(d+2)}{d}, \frac{2d}{d-2} \right),\\
&\frac{(d-2)(pd-2d-4)}{4(d-1)},&p\in&\left[ \frac{2d}{d-2},6\right),\\
&2d-2,&p\in&\left[6,\infty\right),\end{aligned}\right.\]
\end{thm}
This theorem is the combination of theorems~\ref{mainthm1},~\ref{mainthm2} and~\ref{mainthm3}, which are proved in sections~\ref{smt1},~\ref{smt2} and~\ref{smt3} respectively.

\subsection{Notations} 
\label{notations}
The Fourier transform of a function $f$ defined on $\mathbb{T}^d\times\mathbb{R}$ 
is
\[\widehat{f}(\tau,k)=\mathcal{F}_{x,t}f(\tau,k)=\int_{\mathbb{T}^d\times\mathbb{R}}e^{-2\pi i(k\cdot x+\tau t)}f(t,x)\,\mathrm{d}x\mathrm{d}t\] for $k\in\mathbb{Z}^d$ and $\tau\in\mathbb{R}$. 

The function $\chi(z)$ is a smooth, even, nonnegative function that equals $1$ for $|z|\leq 1$ and equals $0$ for $|z|\geq 2$. 

We write $A\lesssim B$ if $A\leq CB$ for some constant $C$; and $A\lesssim_a B$ if the constant $C$ depends on a parameter $a$: $A \leq C(a) B$. Finally, $A \sim B$ if $A \lesssim B$ and $B \lesssim A$

When we fix a scale $N$, we write $A\preceq B$ if $A\leq_\epsilon C_\epsilon N^{\varepsilon}B$.

For a real number $n$, we denote $\| n \|$ for the smallest distance from $n$ to an integer.

\section{Optimality for $p>6$} 
\label{smt0}

\begin{thm} \label{mainthm0} Assume $p>6$, and that $\beta_2, \dots , \beta_d$ satisfy~\eqref{dioph2}. Define $f$ by its Fourier transform
$$
\widehat{f}(k) = \chi \left( \frac{k_1}{N} \right) \dots \chi \left( \frac{k_d}{N} \right)
$$
(so that in particular $\operatorname{Supp} \widehat{f} \subset B(0,N)$). Then for any $\eta>0$, and $N$ sufficiently big,
\begin{equation}
\label{oriole}
\| e^{it \Delta_\beta} f \|_{L^p([0,T] \times \mathbb{T}^d)} \sim N^{\frac{d}{2} - \frac{3d}{p}} T^{\frac{1}{p}} \| f \|_{L^2(\mathbb{T}^d)}
\qquad \mbox{if $T > N^{2d-2+\eta}$}.
\end{equation}
\end{thm}

\subsection*{Proof of Theorem~\ref{mainthm0}} First notice that it suffices to prove that~\eqref{oriole} holds for $p\geq 6$ an even integer, and for
\begin{equation}
\label{blackbird1}
f(x) = h(x_1) \dots h(x_d)
\end{equation}
if $h$ is a function on $\mathbb{T}$ such that
\begin{equation}
\label{blackbird2}
\operatorname{Supp} \widehat h \subset B(0,N) \qquad \mbox{and} \qquad \| e^{i (2\pi)^{-1} t \partial_{x}^2} f \|_{L^p([0,1] \times \mathbb{T})} \sim N^{\frac{1}{2} - \frac{3}{p}}.
\end{equation}
Indeed, $\sqrt N \chi(Nx)$ satisfies this latter condition for all $p \geq 6$. The statement of the theorem for all $p \geq 6$ follows by interpolation.

\bigskip

\noindent \underline{Step 1: the expression for the Strichartz norm.} Consider $f$ and $h$ as above (equations~\eqref{blackbird1} and~\eqref{blackbird2}); and normalize furthermore $\| h \|_{L^2} = 1$.

In order to take advantage of the tensorial definition of $f$, let $F(t) = \| e^{i(2\pi)^{-1} t \partial_x^2} h \|^p_{L^p(\mathbb{T})}$, which is a $1$-periodic function. Assuming that $T \in \mathbb{N}$, the Strichartz norm of $f$ can be written
\begin{align*}
\| e^{it \Delta_\beta} f \|_{L^p([0,T] \times \mathbb{T}^d)}^p & = \int_0^T F(t) F(\beta_2 t) \dots F(\beta_d t)\,\mathrm{d}t \\
& = T \int_0^1 F(t) \underbrace{ \frac{1}{T} \sum_{n=0}^{T-1} F(\beta_2 t +  \beta_2 n ) \dots F(\beta_d t + \beta_d n )}_{\displaystyle G(t)}\, \mathrm{d} t.
\end{align*}

\noindent \underline{Step 2: Fourier series expansion of $F$.} Expand $F$ in Fourier series:
$$
F(t) = \sum_{k \in \mathbb{Z}} a_k e^{2\pi i k t}.
$$
First,
$$
a_0 = \int_0^1 F(t)\,\mathrm{d}t = \| e^{i(2\pi)^{-1} t \partial_x^2} h \|^p_{L^p([0,1] \times \mathbb{T})} \sim N^{\frac{p}{2}-3}.
$$
Second, by H\"older's inequality, 
\[|\partial_tF(t)|\leq p\int_{\mathbb{T}}\big|e^{i(2\pi)^{-1}t\partial_x^2}h(x)\big|^{p-1}\big|\partial_te^{i(2\pi)^{-1}t\partial_x^2}h(x)\big|\,\mathrm{d}x\lesssim F(t)^{\frac{p-1}{p}} \big\|\partial_x^2e^{i(2\pi)^{-1}t\partial_x^2}h(x)\big\|_{L^p}\lesssim N^2 F(t),\] since $\widehat{h}$ is supported at frequencies $\leq N$. Similar bounds can be obtained for higher order derivatives of $F$, leading to the estimate, for all $n \geq 0$,
\begin{equation}
\label{boundak}
|a_k| \lesssim_n a_0 \frac{N^{2n}}{|k|^n}.
\end{equation}

\noindent \underline{Step 3: convergence of $G$ to $a_0^{d-1}$.} Using the Fourier expansion of $F$, $G$ can be written
$$
G(t) = \frac{1}{T} \sum_{k_1 \dots k_d} a_{k_2}\dots a_{k_d} e^{2\pi i (k_2 \beta_2 + \dots + k_d \beta_d)t} \frac{1 - e^{2\pi i (T-1)(k_2 \beta_2 + \dots + k_d \beta_d )}}{1 - e^{2\pi i (k_2 \beta_2 + \dots + k_d \beta_d )}}
$$
so that
$$
|G(t) - a_0^{d-1}| \leq \frac{1}{T} \sum_{(k_2,\dots,k_d) \neq (0,\dots,0)} |a_{k_2}\dots a_{k_d}| \frac{1 }{\| k_2 \beta_2 + \dots + k_d \beta_d \|}.
$$
To bound the above right-hand side, we split it into two pieces: assume first that one of the $|k_i|$ is $> N^{100d}$, for instance $|k_2| > N^{100d}$ while $|k_3| + |k_4| + \dots + |k_d| < N^{100d}$. The corresponding contribution is bounded by
$$
\dots \lesssim \frac{a_0^{d-1}}{T} \sum_{|k_2| > N^{100 d}} \frac{N^{8d}}{|k_2|^{4d}}\left( \sum_{|k_3| < N^{100 d}} \frac{N^2}{|k_3|} \right)^{d-2} |k_2|^{d+1} \leq \frac{a_0^{d-1}}{T},
$$
where we used in the first inequality the bounds~(\ref{boundak}) as well as the Diophantine condition~\eqref{dioph1}.

We are left with the sum over $|k_2| + \dots + |k_d| \lesssim N^{100d}$, which, using once again the bound~(\ref{boundak}), is less than
$$
\frac{(a_0 N^2)^{d-1}}{T} \sum_{|k_2| + \dots + |k_d| \lesssim N^{100d}} \frac{1}{|k_2|\dots|k_d|} \frac{1}{\| \beta_2 k_2 + \dots + \beta_d k_d \|}.
$$

\bigskip

\noindent \underline{Step 4: Proof of $\sum_{|k_2| + \dots + |k_d| \lesssim N^{100d}} \frac{1}{|k_2|\dots|k_d|} \frac{1}{\| \beta_2 k_2 + \dots + \beta_d k_d \|} \preceq 1$.} For $i = (i_2,\dots,i_d) \in \mathbb{N}^d$, $j \in \mathbb{N}$, let
$$
E_{ij} = \{ (k_2,\dots,k_d) \;:\; |k_{r}| \sim 2^{i_r} \; \mbox{and} \; \| \beta_2 k_2 + \dots + \beta_d k_d \| \sim 2^{-j}.
$$By (\ref{dioph1}), we must have \[j\leq i_2+\cdots+i_d+O(\log \log N),\] thus we can decompose\[j=\sum_{r=2}^d j_r,\quad j_r\leq i_r+O(\log N \log N).\]
Proceeding as in Step 3 of Section 2 and using (\ref{dioph1}) again, we see that for any two elements $k$ and $k'$ of $E_{ij}$ we must have $|k_r-(k')_r| \succeq 2^{-j_r}$ for at least one $r$ (if $2^{i_r} < N^{100d}$). By decomposing the box $\prod_{r=2}^d[-2^{i_r},2^{i_r}]$ into small boxed of size $2^{j_2}\times\cdots 2^{j_d}$, we see that
$$
\# E_{ij} \preceq 2^{i_2 + \dots + i_{d} - j} \preceq 1.
$$
But then
$$
\sum_{|k_2| + \dots + |k_d| \lesssim N^{100d}} \frac{1}{|k_2|\dots|k_d|} \frac{1}{\| \beta_2 k_2 + \dots + \beta_d k_d \|} \leq \sum_{i,j} 2^{-i_2 -\dots - i_d + j} \# E_{ij} \preceq 1.
$$

\bigskip

\noindent \underline{Step 5: conclusion.} Gathering the previous estimates,
$$
|G(t) - a_0^{d-1} | \preceq \frac{N^{2d-2}}{T} a_0^{d-1}.
$$
Choosing $T > N^{2d-2+\epsilon}$, we obtain $|G(t) - a_0^{d-1} | < \frac{1}{10} a_0^{d-1}$. But then
$$
\| e^{it \Delta_\beta} f \|_{L^p([0,T] \times \mathbb{T}^d)}^p T = \int_0^1 F(t)  G(t)\, \mathrm{d} t \sim T a_0^{d-1} \int_0^1 F(t)\, \mathrm{d} t \sim T a_0^{d},
$$
which is the desired result.

\section{Proof of the conjecture for $p>6$}
\label{smt1}

\begin{thm} \label{mainthm1}
For $\beta_2, \dots, \beta_d$ satisfying~\eqref{dioph1}, there holds for $p>6$
$$
\| e^{i t \Delta_\beta} f \|_{L^p([0,T] \times \mathbb{R}^d)} \preceq N^{\frac{d}{2} - \frac{d+2}{p}} \left[ 1 + \left( \frac{T}{N^{2d-2}} \right)^{\frac{1}{p}} \right] \| f \|_{L^2} \qquad\textrm{if }\operatorname{Supp} \widehat{f} \subset B(0,N).
$$
\end{thm}

\subsection*{Proof of Theorem~\ref{mainthm1}} \noindent \underline{Step 0: preliminaries.} 
It suffices to prove that
$$
\big\|e^{it\Delta_{\beta}} f\big\|_{L_{t,x}^{6}([0,N^{2d-2}]\times\mathbb{T}^d)}^6 \preceq N^{2d-2}\|f\|_{L^2}^6.
$$
Indeed, if $p=6$, the estimate for $T<N^{2d-2}$ follows immediately, while for $T>N^{2d-2}$ it suffices to add up the above estimate on intervals of size $\sim N^{2d-2}$. Finally, the result follows for $p>6$ by interpolation with the trivial $p=\infty$ case.

Assuming that $f$ is supported in Fourier on $B(0,N)$, expand in Fourier series
\begin{align*}
& f(x) = \sum_{k \in \mathbb{Z}^d} a_k e^{2\pi i k \cdot x} \quad \mbox{where $a_k = 0$ for $|k|>N$}\\
& e^{it\Delta_\beta} f(x) = \sum_{k \in \mathbb{Z}^d} a_k e^{2\pi i ( k \cdot x - tQ(k))}.
\end{align*}
We may normalize $f$ so that \[\sum_{|k|\leq N}|a_k|^2=1.\]Finally, we write
\begin{align*}
& \big\|e^{it\Delta_{\beta}} f\big\|_{L_{t,x}^{6}([0,N^{2d-2}]\times\mathbb{T}^d)}^6 \\
& \qquad = \sum_{\substack{k_i\in\mathbb{Z}^d,|k_i|\leq N;\\k_1+k_3+k_5=k_2+k_4+k_6}}a_{k_1}\overline{a_{k_2}}a_{k_3}\overline{a_{k_4}}a_{k_5}\overline{a_{k_6}}\int_0^{N^{2d-2}}e^{-2\pi i[Q(k_1)-Q(k_2)+Q(k_3)-Q(k_4)+Q(k_5)-Q(k_6)]t}\,\mathrm{d}t,
\end{align*}

\bigskip

\noindent \underline{Step 1: decomposition in $\Lambda_A$.}
For a dyadic number $A\in[N^{2-2d},100 N^{2}]$, define the set\[\begin{aligned}\Lambda_A=\big\{(k_1,\cdots,k_6)&\in(\mathbb{Z}^d)^6:|k_i|\leq N,k_1+k_3+k_5=k_2+k_4+k_6,\\& A \leq \big|Q(k_1)-Q(k_2)+Q(k_3)-Q(k_4)+Q(k_5)-Q(k_6)\big| <2 A\big\}.\end{aligned}\] 
When $A$ is the smallest dyadic number larger than $N^{2-2d}$, the lower bound $A \leq$ above is removed. Note that for $(k_1,\cdots,k_6)\in \Lambda_A$ one has\[\bigg|\int_0^{N^2}e^{-2\pi i[Q(k_1)-Q(k_2)+Q(k_3)-Q(k_4)+Q(k_5)-Q(k_6)]t}\,\mathrm{d}t\bigg|\lesssim A^{-1},\]
thus we only need to show that
\begin{equation}\label{mainest4}\sum_{(k_1,\cdots,k_6)\in \Lambda_A}|a_{k_1}a_{k_2}a_{k_3}a_{k_4}a_{k_5}a_{k_6}|\preceq N^{2d-2} A.\end{equation} 

\bigskip

\noindent \underline{Step 2: decomposition in $\Sigma_{X_1, \dots, X_d}$.} For $(X_1, \dots, X_d)$ in $\mathbb{Z}^d$, let
\[\begin{aligned}\Sigma_{X_1 ,\dots, X_d}=\big\{(k_1,\cdots,k_6)&\in(\mathbb{Z}^d)^6:|k_i|\leq N,k_1+k_3+k_5=k_2+k_4+k_6,\\&(k_1^i)^2-(k_2^i)^2+(k_3^i)^2-(k_4^i)^2+(k_5^i)^2-(k_6^i)^2=X_i \mbox{ for $1 \leq i \leq d$}.\big\}.\end{aligned}\]
Write then
$$
\sum_{(k_1,\cdots,k_6)\in \Lambda_A}|a_{k_1}a_{k_2}a_{k_3}a_{k_4}a_{k_5}a_{k_6}| = \sum_{\substack{(X_1, \dots , X_d)\in \mathbb{Z}^d, \, |X_i| \lesssim N^2 \\ |X_1 + \beta_2 X_2 + \dots + \beta_d X_d|\sim A }} \sum_{(k_1, \dots ,k_6) \in \Sigma_{X_1, \dots, X_d}} |a_{k_1}a_{k_2}a_{k_3}a_{k_4}a_{k_5}a_{k_6}|.
$$
Therefore, to prove~\eqref{mainest4}, it suffices to show that
\begin{subequations}
\begin{align}
\label{bound1}
& \# \{ (X_1,\dots,X_d) \in \mathbb{Z}^d \; : \; |X_i| \lesssim N^2 \mbox{ and } |X_1 + \beta_2 X_2 + \dots + \beta_d X_d| \sim A \} \preceq N^{2d-2} A \\
\label{bound2}
&\sum_{(k_1,\cdots,k_6)\in \Sigma_{X_1, \dots, X_d}}|a_{k_1}a_{k_2}a_{k_3}a_{k_4}a_{k_5}a_{k_6}|\preceq 1\quad \mbox{for fixed $(X_1,\dots,X_d)$}.
\end{align}
\end{subequations}

\bigskip
\noindent \underline{Step 3: proof of the bound~\eqref{bound1}} We will actually prove that, if $K^{d-1} A \gtrsim 1$,
\begin{equation}
\label{boundX}
\# \{ (X_1,\dots,X_d) \in \mathbb{Z}^d  \; : \; |X_i| < K \mbox{ and } |X_1 + \beta_2 X_2 + \dots + \beta_d X_d|<A\} \preceq K^{d-1} A.
\end{equation}
If $A>1$, this is trivial: one can choose freely $X_1 \dots X_{d-1}$, and then at most $\sim A$ choices for $X_d$ are allowed.

Assume now that $A<1$, and that $X_1,\dots ,X_d$ and $X_1',\dots, X_d'$ satisfy
\begin{align*}
& | X_1 + \beta_2 X_2 + \dots + \beta_d X_d | < A \\
& | X_1' + \beta_2 X_2' + \dots + \beta_d X_d' | < A.
\end{align*}
But then
$$
A \gtrsim | (X_1-X_1') + \beta_2 (X_2-X_2') + \dots + \beta_d (X_d-X_d') | \succeq \frac{1}{(|X_2 - X'_2| + \dots + |X_d - X_d'|)^{d-1}},
$$
where the last inequality follows from the Diophantine condition~\eqref{dioph1}. This means that 
$$|(X_2\dots X_d) - (X'_2 \dots X'_d)| \succeq A^{-\frac{1}{d-1}},$$
or in other words that the density of admissible coordinates in $(X_2,\dots, X_d)$ is bounded by $A^{-1}$. Since $A<1$, $X_1$ is completely determined by $(X_2,\dots,X_d)$ and the desired bound~\eqref{boundX} follows.

\bigskip

\noindent \underline{Step 4: proof of the bound~\eqref{bound2}} By the Cauchy-Schwarz inequality,
$$
\sum_{(k_1,\cdots,k_6)\in \Sigma_{X_1 \dots X_d}}|a_{k_1}a_{k_2}a_{k_3}a_{k_4}a_{k_5}a_{k_6}| \lesssim \sum_{(k_1,\cdots,k_6)\in \Sigma_{X_1\dots X_d}}|a_{k_1}a_{k_3}a_{k_5}|^2+\sum_{(k_1,\cdots,k_6)\in \Sigma_{X_1 \dots X_d}}|a_{k_2}a_{k_4}a_{k_6}|^2.
$$
By symmetry it suffices to estimate the first sum, which is bounded by\[\bigg(\sum_{k\in\mathbb{Z}^d,|k|\leq N}|a_k|^2\bigg)^3\cdot\sup_{k_1,k_3,k_5}\#\big\{(k_2,k_4,k_6):(k_1,\cdots,k_6)\in\Sigma_{X_1 ,\cdots ,X_d}\big\},\] so that we only need to show, for fixed $X_1, \dots ,X_d$ and $k_1,k_3,k_5$, that there are $\preceq 1$ choices for $(k_2,k_4,k_6)$ such that $(k_1,\cdots k_6)\in\Sigma_{X_1,\dots,X_d}$.

Now, if $X_1 \dots X_d$ and $k_1,k_3,k_5$ are given, and $(k_1,\cdots k_6)\in\Sigma_{X_1,\dots,X_d}$, then \[k_2^1+k_4^1+k_6^1=M,\quad \mbox{and} \quad (k_2^1)^2+(k_4^1)^2+(k_6^1)^2=S\] are both fixed. Thus\[(3k_2^1-M)^2+(3k_4^1-M)^2+(3k_2^1-M)(3k_4^1-M)=\frac{9S-3M^2}{2}\] is also fixed (and is nonzero, unless $k_2^1=k_4^1=k_6^1=M/3$). 

Denoting $a = 3k_2^1-M$ and $b = 3k_4^1-M$, this can be written
$$
\frac{9S - 3M^2}{2} = a^2 + b^2 + ab = (a - \omega b) (a - \overline{\omega} b), \quad \mbox{with $\omega = e^{\frac{2\pi i}{3}}$}.
$$
By the divisor estimate in $\mathbb{Z}[\omega]$, there are $\preceq 1$ choices for $(k_2^1,k_4^1,k_6^1)$. A similar argument works for $(k_2^i,k_4^i,k_6^i)$, with $2 \leq i \leq d$, so that the bound for the number of choices for $(k_2,k_4,k_6)$ is proved, which completes the proof.

\section{A first bound for $p>\frac{2(d+2)}{d}$}
\label{smt2}

\begin{thm}
\label{mainthm2} Assume $\beta_2, \dots, \beta_d$ are chosen generically. Then if $p> \frac{2(d+2)}{d}$, and $\operatorname{Supp} \widehat{f} \subset B(0,N)$, 
$$
\| e^{it\Delta_\beta} f \|_{L^p([0,T] \times \mathbb{T}^d)} \preceq N^{\frac{d}{2}-\frac{d+2}{p}} \left( 1 + \left( \frac{T}{N^{\theta_1(p)}} \right)^{1/p} \right) \| f \|_{L^2} 
\quad \mbox{with} \quad
\left\{
\begin{array}{l}
\displaystyle \theta_1(p) = 2(d-1) \frac{p-p^*}{p+8-p^*} \\ \displaystyle  p^* = \frac{2(d+2)}{d}.
\end{array}
\right.
$$
\end{thm}

\begin{rem}
Conjecture \ref{longstr} gives the exponent $\theta(p)=\frac{d}{2}(p-p^*)$ for $p^* < p < 6$.
\end{rem}

\subsection{A generic Diophantine property}

\begin{lem} 
For generic $\beta_2,\dots,\beta_d$ in $[1,2]$, there exists a constant $C_0$ such that, for all $a=(a_i)\in\mathbb{Z}^d$ and $b=(b_i)\in\mathbb{Z}^d$, 
\begin{equation}
\prod_{i=2}^d\left|\beta_i - \frac{a_ib_1}{a_1b_i}\right|\geq C_0 \prod_{i=1}^{d}\left[(1+|a_i|)^{-1}(\log(2+|a_i|))^{-d}\cdot (1+|b_i|)^{-1}(\log(2+|b_i|))^{-d}\right],
\label{dioph3}
\end{equation} given that $a_1\neq 0$ and $b_i\neq 0$ for $2\leq i\leq d$.
\end{lem}

\begin{proof} 
Since the left hand side of (\ref{dioph3}) is never zero given that all $\beta_i(2\leq i\leq d)$ are irrational, (\ref{dioph3}) will hold true if there is some dyadic number $M\geq 1$ such that
\begin{equation}(\beta_2,\cdots,\beta_{d})\not\in Q_M,\end{equation} where the set $Q_M$ is defined by
\begin{equation}\begin{split}Q_M:=\bigg\{(\beta_2,\cdots,\beta_d)\in[1,2]^{d-1}:\,\,&\text{(\ref{dioph2}) is false with $C_0$ replaced by $1$, for some}\\
&\text{$(a_i)$ and $(b_i)$ such that $a_1b_i\neq 0$ for all $2\leq i\leq d$,}\\
&\text{and $|a_i|\geq M$ or $|b_i|\geq M$ for at least one $1\leq i\leq d$}\bigg\}
\end{split}\end{equation} 
Now we shall prove that $|Q_M|\to 0$ as $M\to\infty$, which, by the Borel-Cantelli theorem, clearly implies (\ref{dioph3}). In fact, by elementary calculus one has that
\begin{equation}\left|\left\{(\beta_2,\cdots,\beta_d)\in[1 ,2]^{d-1}:\prod_{i=2}^{d}|\beta_i-y_i|\leq\varepsilon\right\}\right|\lesssim \varepsilon(\log(1/\varepsilon))^{d-2} \end{equation} for any $y=(y_i)\in\mathbb{R}^{d-1}$ and $\varepsilon\leq 1/2$. This gives
\begin{equation}\label{summation}|Q_M|\lesssim\sum_{A\gtrsim M;\text{ dyadic}}(\log A)^{d-2}\sum_{(a_i),(b_i)}^*\prod_{i=1}^{d}\left[(1+|a_i|)^{-1}(\log(2+|a_i|))^{-d}\cdot (1+|b_i|)^{-1}(\log(2+|b_i|))^{-d}\right],\end{equation} where the second summation is restricted to the set where the maximum of $|a_i|$ and $|b_i|$ is $\sim A$. Evaluating the sum in (\ref{summation}), one gets \[|Q_M|\lesssim\sum_{A\gtrsim M;\text{ dyadic}}(\log A)^{-2}\lesssim(\log M)^{-1},\] as desired.
\end{proof} 

\subsection{Bounds on the fundamental solution}

We will denote $K_N$ the fundamental solution of $i\partial_t + Q(D)$ smoothly truncated to frequencies $\lesssim N$. More precisely, set
$$
K_N(t,x) = \sum_{k \in \mathbb{Z}^d} e^{i 2\pi (x\cdot k - tQ(k))} \chi \left( \frac{k_1}{N} \right) \dots \chi \left( \frac{k_d}{N} \right) 
$$
(recall that $\chi$ is a smooth, nonnegative function, supported on $B(0,2)$, and equal to $1$ on $B(0,1)$).

\begin{lem}[Dispersive bound] \label{dispersive}
If $|t| \lesssim \frac{1}{N}$,
$$
\forall y \in \mathbb{R}, \qquad \bigg|\sum_{k\in\mathbb{Z}}e^{2\pi i(yk+k^2t)}\chi\bigg(\frac{k}{N}\bigg)\bigg|\lesssim \min \left( N,\frac{1}{\sqrt{t} } \right).
$$
\end{lem}

\begin{lem}[Weyl bound] \label{weyl0} 
If $a\in\mathbb{Z} \setminus \{0\}$, $q \in \{ 1,\dots, N \}$, $(a,q)=1$ and $\left|t-\frac{a}{q}\right|\leq \frac{4}{Nq}$:
\begin{equation}\label{weyl} \forall y \in \mathbb{R}, \qquad \bigg|\sum_{k\in\mathbb{Z}}e^{2\pi i(yk+k^2t)}\chi\bigg(\frac{k}{N}\bigg)\bigg|\lesssim\frac{N}{\sqrt{q}\left(1+N\bigg|t-\frac{a}{q}\bigg|^{1/2}\right)}.\end{equation} 
\end{lem}
\begin{proof} See for instance Bourgain \cite{B1}, Lemma 3.18.\end{proof}

A consequence of the Weyl bound and of the genericity of the $(\beta_i)$ is the following pointwise bound on $K_N(t,x)$.

\begin{prop}\label{boundk} Assume that $\beta =(\beta_2,\cdots,\beta_d)$ is chosen generically. Then for $\frac{2}{N}< t < N^K$, for a constant $K$, we have 
\begin{equation}\label{bound}|K_N(t,x)|\preceq N^{\frac{d+1}{2}}t^{\frac{1}{4}}.\end{equation}  
\end{prop}
\begin{proof} \underline{The case d=2.} By Dirichlet's lemma, there exists $a,a' \in \mathbb{Z}$ and $q,q' \in \{ 1 \dots N \}$ such that $\delta = \left| t - \frac{a}{q} \right| \leq \frac{1}{Nq}$ and $\delta' = \left| \beta t - \frac{a'}{q'} \right| \leq \frac{1}{Nq'}$. By the Weyl bound above, $K_N$ can be bounded by
\begin{equation}
\label{pinguin}
\left| K_N(t,x) \right| \lesssim \frac{N^2}{\sqrt{qq'}\left(1 + N \left|t - \frac{a}{q}\right|^{1/2} \right) \left( 1 + N \left|\beta t - \frac{a'}{q'}\right|^{1/2} \right)}.
\end{equation}
The Diophantine condition~(\ref{dioph1}) implies that $\delta + \delta' \succeq \frac{1}{tq^2 (q')^2}$. Inserting this bound in the above gives 
$$
\left| K_N(t,x) \right| \lesssim \frac{N^2}{\sqrt{qq'}(1 + N\sqrt{\delta})(1+ N \sqrt{\delta'} )} \preceq \frac{N^2}{\sqrt{qq'} \left( 1 + \frac{N}{\sqrt{t} q q' }\right)}
$$
Observe that the map $y \mapsto \sqrt{y} \left( 1 + \frac{\alpha}{y} \right)$ reaches its minimum, equal to $2 \sqrt{\alpha}$, when $y = \alpha$. This implies that the above right-hand side is maximum for $qq' = \frac{N}{\sqrt{t}}$, leading to the bound
$$
\left| K_N(t,x) \right| \preceq N^{3/2} t^{1/4}.
$$

\noindent \underline{The general case $d \geq 2$.} By Dirichlet's lemma, for each $1 \leq i \leq d$ there are integers $a_i \in \mathbb{Z}$, $q_i \in \{ 1,\dots, N\}$ such that $(a_i,q_i) = 1$ and
$$
\left| \beta_i t - \frac{a_i}{q_i} \right| \leq \frac{1}{Nq_i}.
$$ Since $|t|>2/N$ we have $a_i\neq 0$. Let $|q_i|\sim Q_i$ and $|\beta_it-a_i/q_i|\sim K_i$. Then by Lemma \ref{weyl0},
\[|K_N(t,x)|\lesssim(Q_1\cdots Q_d)^{-1/2}N^d\prod_{i\geq 1:K_i\geq 1/N^2}(N^{-1}K_i^{-1/2}).\] 
We may assume without loss of generality that $K_1\sim\min_i K_i$, then\[\left|\beta_i-\frac{a_iq_1}{a_1q_i}\right|\lesssim\frac{K_i}{t}\] for $2\leq i\leq d$. Since $|a_i|\sim tQ_i$, by generic Diophantine condition (\ref{dioph3}) gives \[\prod_{i=1}^{d}Q_i\cdot\prod_{i=1}^d(TQ_i)\cdot\prod_{i=2}^d(t^{-1}K_i)\succeq 1.\] This gives \[\prod_{i=1}^dQ_i\succeq t^{-1/2}\prod_{i=2}^dK_i^{-1/2},\] and thus \[|K_N(t,x)|\preceq t^{1/4}N^d \prod_{i=2}^dK_i^{1/4}\prod_{i\geq 1:K_i\geq 1/N^2}(N^{-1}K_i^{-1/2}).\] This final expression is maximized when each $K_i\sim N^{-2}$, and gives\[|K_N(t,x)|\preceq N^{\frac{d+1}{2}}t^{\frac{1}{4}},\]
which concludes the proof.
\end{proof} 

In the case when $|t|$ is extremely small, we use instead the following bound.
\begin{prop}\label{boundsm} For any $A>0$,
\begin{equation}\label{boundsm0}\bigg\|\mathcal{F}_{t,x}\bigg[\chi\bigg(\frac{t}{A}\bigg)K_N(t,x)\bigg]\bigg\|_{L^{\infty}}\lesssim A.\end{equation}
\end{prop}
\begin{proof} Note that \[\mathcal{F}_{t,x}\bigg[\chi\bigg(\frac{t}{A}\bigg)K_N(t,x)\bigg](k,\tau)=A\chi\bigg(\frac{k_1}{N}\bigg)\dots \chi\bigg(\frac{k_d}{N}\bigg)\widehat{\chi}(A(\tau+Q(k))),\] the result follows.
\end{proof}

\subsection{Proof of Theorem~\ref{mainthm2}}
To prove Theorem \ref{mainthm2}, we only need to show that \begin{equation}
\| e^{it\Delta_\beta} f \|_{L^p([0,N^{\theta_1(p)}] \times \mathbb{T}^d)} \preceq N^{\frac{d}{2}-\frac{d+2}{p}}\| f \|_{L^2}, \end{equation} since the theorem follows then by iterating on time intervals of length $N^{\theta_1(p)}$.

\bigskip
\noindent
\underline{Step 1: decomposition of the kernel} Let $\phi$ be a smooth, real, non-negative function supported on $B(0,2)$ such that $\phi > 1$ on $B(0,1)$ and $\widehat{\phi} \geq 0$. For a number $A \in (0,\frac{1}{N})$ to be fixed later, decompose $\phi \left(\frac{t}{T}\right)K_N(t,x) $ into
\begin{align*}
\phi \left(\frac{t}{T}\right) K_N(t,x) = 
\underbrace{ \phi \left(\frac{t}{T}\right)\chi \left( \frac{t}{A} \right)K_N(t,x)}_{\displaystyle J_1(t,x)} 
& + \underbrace{\phi \left(\frac{t}{T}\right) \chi(Nt) \left[ 1- \chi \left( \frac{t}{A}  \right)\right]K_N(t,x)}_{\displaystyle J_2(t,x)} \\
& + \underbrace{\phi \left(\frac{t}{T}\right) \left[ 1 - \chi(Nt) \right]K_N(t,x)}_{\displaystyle J_3(t,x)}.
\end{align*}
Using lemmas~\ref{dispersive}, \ref{boundk} and~\ref{boundsm}, we obtain
\begin{align*}
& \| \widehat{J_1} \|_{L^\infty} \lesssim A \\
& \| J_2 \|_{L^\infty} \lesssim \frac{1}{A^{\frac{d}{2}}} \\
& \| J_3 \|_{L^\infty} \preceq T^{\frac{1}{4}} N^{\frac{d+1}{2}}.
\end{align*}

\bigskip \noindent
\underline{Step 2: level set estimates.} We essentially follow the argument in Bourgain~\cite{B1}, which is a modification adapted to level set estimates of the Stein-Tomas argument~\cite{Tomas}. Start with $f \in L^2 (\mathbb{T}^d)$ supported in Fourier on $B(0,N)$ and of norm 1: $\|f\|_{L^2(\mathbb{T}^d)} = 1$. Setting $F = e^{it \Delta_\beta} f$, we want to estimate the size of 
$$
E_\lambda = \{ (x,t) \in \mathbb{T}^d \times [-T,T] \; : \; |F(x,t)| > \lambda \},
$$
for a time $T \geq 1$ yet to be fixed. Setting $\widetilde{F} = \frac{F}{|F|} \mathbbm{1}_{E_\lambda}$, we can bound, using successively Plancherel's theorem, the Cauchy Schwarz inequality, Plancherel's theorem again, and finally $\|f\|_2 = 1$,
\begin{align*}
\lambda^2 &|E_\lambda|^2  \lesssim \left[ \int_{\mathbb{T}^2 \times \mathbb{R}} \widetilde{F}(x,t) \overline{F(x,t) \phi \left( \frac{t}{T} \right)} \,\mathrm{d}x\,\mathrm{d}t \right]^2 \\
& = \left[ \sum_{k} \int \widehat{ \widetilde{F}}(\tau,k) \overline{T \widehat{f}(k) \widehat{\phi}(T(\tau + Q(k)))} \chi \left( \frac{k_1}{N} \right)^{1/2} \dots \chi \left( \frac{k_d}{N} \right)^{1/2} \,\mathrm{d}\tau \right]^2 \\
& \leq \left[ \sum_{k \in \mathbb{Z}^d} \int_{\mathbb{R}} \left| \widehat{\widetilde{F}} (\tau,k) \right|^2 T \widehat{\phi}(T(\tau + Q(k))) \chi \left( \frac{k_1}{N} \right) \dots \chi \left( \frac{k_d}{N} \right)  \,\mathrm{d}\tau \right] \left[ \sum_k |\widehat{f}(k)|^2 \int T \widehat{\phi}(T(\tau + Q(k))) \,\mathrm{d}\tau \right] \\
& \lesssim \sum_k \int \left| \widehat{\widetilde{F}} (\tau,k) \right|^2 T \widehat{\phi} (T(\tau + Q(k)))\chi \left( \frac{k_1}{N} \right) \dots \chi \left( \frac{k_d}{N} \right) \,\mathrm{d}\tau.
\end{align*}
Applying once more Plancherel's theorem, the above gives
\begin{align*}
\lambda^2 |E_\lambda|^2 & \lesssim \sum_k \int \left| \widehat{\widetilde{F}} (\tau,k) \right|^2 T \widehat{\phi}(T(\tau + Q(k)))\chi \left( \frac{k_1}{N} \right) \dots \chi \left( \frac{k_d}{N} \right) \,\mathrm{d}\tau \\
& = \int \left[ \left( K_N \phi \left(\frac{.}{T}\right) \right) * \widetilde{F} \right] (t,x) \overline{\widetilde{F}(t,x)} \,\mathrm{d}x\,\mathrm{d}t.
\end{align*}
Now using the decomposition of Step 1,
\begin{align*}
\lambda^2 |E_\lambda|^2 & \lesssim \left< (J_1 + J _2 + J_3) * \widetilde{F}\,,\,\widetilde{F} \right> \\
& \lesssim \|\widehat{J_1}\|_{L^\infty} \|\widetilde{F}\|_{L^2}^2 + \left(\|J_2\|_{L^\infty}  + \|J_3\|_{L^\infty} \right) \|\widetilde{F}\|_{L^1}^2 \\
& \preceq A |E_\lambda| + \left( \frac{1}{A^{\frac{d}{2}}} + T^{\frac{1}{4}} N^{\frac{d+1}{2}} \right) |E_\lambda|^2
\end{align*}
Summarizing, we get if $A < \frac{1}{N}$
\begin{equation}
\label{levelsetestimate}
\lambda^2 |E_\lambda|^2 \preceq A |E_\lambda| + \left( \frac{1}{A^{\frac{d}{2}}} + T^{\frac{1}{4}} N^{\frac{d+1}{2}} \right) |E_\lambda|^2
\end{equation}

\bigskip

\noindent
\underline{Step 3: from level set estimates to $L^p$ bounds.} 
Recall first that $\|F\|_{L^\infty(\mathbb{R} \times \mathbb{T}^d)} \lesssim N^{\frac{d}{2}}$ by the Sobolev embedding theorem. 
Choose next $\delta>0$. When estimating $|E_\lambda|$, two cases have to be distinguished:
\begin{itemize}
\item If $\lambda^2 > T^{\frac{1}{4}} N^{\frac{d+1}{2}+\delta}$, then we choose $A = \frac{N^\delta}{ \lambda^{\frac{4}{d}}}$ (notice that $A < \frac{1}{N}$). The bound~\eqref{levelsetestimate} becomes then
$$
|E_\lambda| \preceq N^\delta \lambda^{-\frac{2(d+2)}{d}}.
$$  
\item If $\lambda^2 < T^{\frac{1}{4}} N^{\frac{d+1}{2}+\delta}$, we rely on the Chebyshev inequality and the estimate $\|F\|_{L^{\frac{2(d+2)}{d}}([-T,T] \times \mathbb{T}^d)} \lesssim T^{\frac{d}{2(d+2)}}$ (which follows from the $L^{\frac{2(d+2)}{d}}$ bound of Bourgain-Demeter~\cite{BD}) to obtain
$$
|E_\lambda| \lesssim T \lambda^{-\frac{2(d+2)}{d}}.
$$
\end{itemize}
All in all, this gives for $p>\frac{2(d+2)}{d}$
\begin{align*}
\|F\|^p_{L^p([-T,T] \times \mathbb{T}^d)} 
& = p \int_{0}^{N^{\frac{d}{2}}} \lambda^{p-1} |E_\lambda|\,d\lambda \\
& \preceq \int_0^{T^{\frac{1}{8}} N^{\frac{d+1}{4}+\frac{\delta}{2} }} T \lambda^{-\frac{2(d+2)}{d} + p-1}\,d\lambda 
+ \int_{T^{\frac{1}{8}} N^{\frac{d+1}{4}+\frac{\delta}{2}}}^{N^{\frac{d}{2}}} N^\delta \lambda^{-\frac{2(d+2)}{d} + p-1}\,d\lambda  \\
& \preceq T \left( T^{\frac{1}{8}} N^{\frac{d+1}{4}+\frac{\delta}{2}} \right)^{-\frac{2(d+2)}{d} + p} + N^{\frac{pd}{2}-(d+2)+\delta}.
\end{align*}
Since the above is true for any $\delta>0$, we get upon choosing $T = N^{\theta_1(p)}$
$$
\|F\|_{L^p([0,N^{\theta_1(p)}] \times \mathbb{T}^d)} \preceq N^{\frac{d}{2}-\frac{d+2}{p}},
$$
from which the desired bound follows immediately.

\section{An improvement if $p>\frac{2(d+2)}{d}$ for $d \geq 4$}
\label{smt3}
\begin{thm}\label{mainthm3}  Assume $\beta_2, \dots, \beta_d$ are chosen generically. Then if $p> \frac{2(d+2)}{d}$, and $\operatorname{Supp} \widehat{f} \subset B(0,N)$,
$$
\| e^{it\Delta_\beta} f \|_{L^p([-T,T] \times \mathbb{T}^d)} \preceq \| f \|_{L^2} N^{\frac{d}{2}-\frac{d+2}{p}} \left( 1 + \left( \frac{T}{N^{\theta_2(p)}} \right)^{1/p} \right)
\quad \mbox{with} \quad
\left\{
\begin{array}{l}
\displaystyle \theta_2(p) = \frac{d^2-2d}{4d-4}(p-p^*) \\ \displaystyle p^* = \frac{2(d+2)}{d}.
\end{array}
\right.
$$
\end{thm}
\begin{rem} Observe that $\theta_2(p)>\theta_1(p)$ if and only if $p>\frac{2d}{d-2}$. Thus Theorem \ref{mainthm2} gives a better bound if $d\geq 4$ and $6>p>2d/(d-2)$.
\end{rem}

The proof will rely on a decomposition into major and minor arcs. Define to that effect, for $Q$ a power of 2,
$$
\Lambda_Q(t) = \sum_{\substack{q \in \mathbb{N}^* \\ Q \leq q < 2Q}} \sum_{\substack{a \in \mathbb{N}^* \\ (a,q) = 1}} \chi \left( NQ\left( t-\frac{a}{q} \right)\right)
$$
and
$$
\rho(t) = 1 - \sum_{\substack{Q \in 2^\mathbb{N} \\ Q < c_0 N}} \Lambda_Q(t),
$$
where $c_0$ is taken sufficiently small, in particular to ensure that the supports of the $\Lambda_Q(t)$ are disjoint for $Q < c_0 N$.

\subsection{A genericity condition on $\beta_2,\dots ,\beta_d$}

Notice that the support of $\Lambda_Q(t)$ has density $\sim \frac{Q}{N}$; therefore, it is natural to expect that for generic $\beta_2, \dots, \beta_d$, the support of $\Lambda_{Q_1}(t) \Lambda_{Q_2}(\beta_2 t) \dots \Lambda_{Q_k}(\beta_d t)$ has density $\sim \frac{Q_1 \dots Q_k}{N^k}$. This is the content of the next lemma.
\begin{lem}
\label{generic}
Generic $\beta_2, \dots, \beta_d \in [1,2]^d$ are such that, for $1 \leq k \leq d$, and $T \geq 1$,
$$
\int_0^T \Lambda_{Q_1}(t) \Lambda_{Q_2}(\beta_2 t) \dots \Lambda_{Q_k}(\beta_k t)\,\mathrm{d}t \preceq \frac{Q_1 \dots Q_k}{N^k} T.
$$
\end{lem}

\begin{proof} Using the Borel-Cantelli lemma, it suffices to prove that
$$
\int_{1}^{2} \dots \int_{1}^{2} \int_0^T \Lambda_{Q_1}(t) \Lambda_{Q_2}(\beta_2 t) \dots \Lambda_{Q_k}(\beta_k t)\,\mathrm{d}t \,\mathrm{d}\beta_2 \dots \mathrm{d}\beta_k \lesssim \frac{Q_1 \dots Q_k}{N^k} T.
$$
Observe first that, due to the definition of $\Lambda_Q$,
$$
\int_{1}^{2} \Lambda_Q(\beta t) \,\mathrm{d}\beta = \frac{1}{t} \int_{t/C_0}^{C_0t} \Lambda_Q(y)\,\mathrm{d}y \lesssim \frac{Q}{N}.
$$
But then by Fubini's theorem
\begin{align*}
& \int_{1}^{2} \dots \int_{1}^{2} \int_0^T \Lambda_{Q_1}(t) \Lambda_{Q_2}(\beta_2 t) \dots \Lambda_{Q_k}(\beta_k t)\,\mathrm{d}t \,\mathrm{d}\beta_2 \dots \mathrm{d}\beta_k \\
& \qquad \qquad = \int_0^T \Lambda_{Q_1}(t) \left[ \int_{1}^{2} \Lambda_{Q_2}(\beta_2 t) \,\mathrm{d}\beta_2 \right] \dots \left[ \int_{1}^{2} \Lambda_{Q_k}(\beta_2 t) \,\mathrm{d}\beta_k \right]\,\mathrm{d}t \\
& \qquad \qquad \lesssim \int_0^t \Lambda_{Q_1}(t) \frac{Q_2 \dots Q_k}{N^{k-1}}\,\mathrm{d}t \lesssim \frac{Q_1 \dots Q_k}{N^k} T.
\end{align*}
\end{proof}

\subsection{Kernel bounds}

Proceeding as in Section~\ref{smt2}, first decompose $K_N$ as follows:
\begin{align*}
\phi\left(\frac t T \right) K_N(t,x) = \underbrace{\phi\left(\frac t T \right) \chi \left(\frac{t}{A} \right) K_N(t,x)}_{\displaystyle J_1(t,x)} & + \underbrace{\phi\left(\frac t T \right) \chi(Nt) \left( 1- \chi \left(\frac{t}{A} \right)\right) K_N(t,x)}_{\displaystyle J_2(t,x)} \\
& + \underbrace{\phi\left(\frac t T \right)(1 - \chi(Nt)) K_N(t,x)}_{\displaystyle J_3(t,x)}.
\end{align*}
Turning to $J_3$,
\begin{align*}
& J_3(t,x) = J_3(t,x) \prod_{j=1}^{d} \left[ \sum_{Q < c_0 N} \Lambda_Q(\beta_j t) + \rho(\beta_j t) \right] \\
& = \sum_{k = 1}^d \sum_{Q_{1}, \dots, Q_{k}}
\underbrace{
\sum_{\substack{\{i_1 ,\dots, i_{d-k}\} \cup \{ j_1 ,\dots, j_k \} = \{1, \dots, d \} \\ \{i_1 ,\dots ,i_{d-k}\} \cap \{ j_1, \dots, j_k \}=\emptyset}}
J_3(t,x) \rho(\beta_{i_1} t) \dots \rho(\beta_{i_{d-k}} t)
\Lambda_{Q_{1}}(\beta_{j_1}(t)) \dots \Lambda_{Q_{k}}(\beta_{j_k}(t))
}_{\displaystyle J_3^{Q_1, \dots, Q_k}(t,x)}
\end{align*}
From Dirichlet's lemma as well as lemmas~\ref{weyl0},~\ref{dispersive} and~\ref{generic}, we conclude that, for generic $\beta_2 \dots \beta_d$, and $A < \frac{1}{N}$,
\begin{align*}
& \| \widehat{J_1} \|_{L^\infty} \lesssim A \\
& \| J_2 \|_{L^\infty} \lesssim \frac{1}{A^{\frac{d}{2}}} \\
& \| J_3^{Q_1 ,\dots, Q_k} \|_{L^\infty} \lesssim \frac{N^{\frac{k}{2} + \frac{d}{2}}}{\sqrt{Q_1 \dots Q_k}} \\
& \| \widehat{J_3^{Q_1 ,\dots, Q_k}} \|_{L^\infty} \lesssim \frac{Q_1 ,\dots, Q_k}{N^k}T
\end{align*}

\subsection{Proof of Theorem~\ref{mainthm3}} By the same argument as in Section~\ref{smt2}, we find
$$
\lambda^2 |E_\lambda|^2 \lesssim A |E_\lambda| + \frac{1}{A^{\frac{d}{2}}} |E_\lambda|^2 + \sum_k \sum_{Q_1, \dots, Q_k} \left< J_3^{Q_1, \dots, Q_k} * \widetilde{F}\,,\, \widetilde{F} \right>.
$$
For $\lambda > C_1 N^{\frac{d}{4}}$ (for a sufficiently big constant $C_1$) we can choose $A = \frac{C_2}{\lambda^{\frac{4}{d}}}$ (for a sufficiently big constant $C_2$), leading to
$$
\lambda^2 |E_\lambda|^2 \lesssim \frac{1}{\lambda^{\frac{4}{d}}} |E_\lambda| + \sum_k \sum_{Q_1 ,\dots, Q_k} \left< J_3^{Q_1, \dots, Q_k} * \widetilde{F}\,,\, \widetilde{F} \right>.
$$
This implies that
$$
\lambda^2 |E_\lambda|^2 \preceq \frac{1}{\lambda^{\frac{4}{d}}} |E_\lambda| + \max_{k,Q_1,\dots ,Q_k} \left< J_3^{Q_1 ,\dots, Q_k} * \widetilde{F}\,,\, \widetilde{F} \right>.
$$
We now fix $k,Q_1, \dots ,Q_k$ for which the $\max$ above is realized, as well as $\delta>0$. Two cases need to be considered
\begin{itemize}
\item If $\frac{N^{\frac{k}{2}+ \frac{d}{2}}}{\sqrt{Q_1 \dots Q_k}}< N^{-\delta} \lambda^2$, we use the $L^\infty$ bound on $J^{Q_1, \dots, Q_k}_3$ and obtain
$$
\lambda^2 |E_\lambda|^2 \preceq \frac{1}{\lambda^{\frac{4}{d}}} |E_\lambda| + N^{-\delta} \lambda^2 |E_\lambda|^2 
$$
which implies
$$
|E_\lambda| \preceq \frac{1}{\lambda^{\frac{2(d+2)}{d}}}.
$$
\item Otherwise, $\frac{N^{\frac{k}{2}+ \frac{d}{2}}}{\sqrt{Q_1, \dots, Q_k}} >\lambda^2 N^{-\delta}$ and the $L^\infty$ bound on $\widehat{J^{Q_1 \dots Q_k}_3}$ gives
$$
\lambda^2 |E_\lambda|^2 \preceq \frac{1}{\lambda^{\frac{4}{d}}} |E_\lambda| + \frac{Q_1 \dots Q_k}{N^k} T |E_\lambda| \preceq \frac{1}{\lambda^{\frac{4}{d}}}|E_\lambda| + \frac{N^{d+2\delta} T}{\lambda^4} |E_\lambda|,
$$
which implies
$$
|E_\lambda| \preceq  \frac{1}{\lambda^{ \frac{2(d+2)}{d}}} + \frac{N^{d+2\delta} T}{\lambda^6}.
$$
\end{itemize}
On the one hand, we deduce, since these inequalities hold for any $\delta>0$, that, for any $\lambda> N^{\frac{d}{4}}$,
$$
|E_\lambda| \preceq \frac{1}{\lambda^{ \frac{2(d+2)}{d}}} + \frac{N^{d} T}{\lambda^6} \lesssim
\left\{
\begin{array}{ll}
\frac{1}{\lambda^{\frac{2(d+2)}{d}}} & \mbox{for $\lambda > N^{\frac{d^2}{4d-4}} T^{\frac{d}{4d-4}}$} \\
\frac{N^{d} T}{\lambda^6} & \mbox{for $\lambda < N^{\frac{d^2}{4d-4}} T^{\frac{d}{4d-4}}$}
\end{array}
\right.
$$
On the other hand, we can still resort to the bound of Bourgain-Demeter
$$
|E_\lambda| \lesssim \frac{T}{\lambda^{\frac{2(d+2)}{d}}}.
$$
Combining these bounds, we obtain the desired estimate: for $\frac{2(d+2)}{d} < p < 6$,
\begin{align*}
& \|F\|_{L^p([-T,T] \times \mathbb{T}^d)} = p \int_0^{N^{\frac{d}{2}}} \lambda^{p-1} |E_\lambda|\,d\lambda \\
& \quad \preceq \int_0^{N^{\frac{d^2}{4d-4}}} T \lambda^{-\frac{2(d+2)}{d}+p-1}\,d\lambda + \int_{N^{\frac{d^2}{4d-4}}}^{N^{\frac{d^2}{4d-4}}T^{\frac{d}{4d-4}}} T N^d \lambda^{p-7} \,d\lambda + \int_{N^{\frac{d^2}{4d-4}}T^{\frac{d}{4d-4}}}^{N^\frac{d}{2}} \lambda^{-\frac{2(d+2)}{d}+p-1}\,d\lambda \\
& \quad \preceq N^{\frac{d}{2} (p - \frac{2(d+2)}{d})} + T N^{\frac{d^2}{4d-4}(p-2\frac{d+2}{d})}.
\end{align*}

\bigskip \bigskip

\noindent \textbf{Acknowledgements}  PG was supported by the National Science Foundation grant DMS-1301380.

\end{document}